\documentclass[12pt,oneside]{amsart}

\usepackage{amssymb}
\usepackage{amsfonts}
\usepackage{amscd}
\usepackage[matrix, arrow, curve]{xy}
\usepackage{graphicx}
\usepackage{color}

\numberwithin{equation}{section}
\newtheorem{theorem}{Theorem}[section]

\newtheorem{proposition}[theorem]{Proposition}
\newtheorem{corollary}[theorem]{Corollary}
\newtheorem{lemma}[theorem]{Lemma}
\theoremstyle{definition}

\theoremstyle{remark}
\newtheorem*{remark}{Remark}
\newtheorem{example}{Example}

\begin{document}
\newcommand{\M}{\mathcal{M}}
\newcommand{\F}{\mathcal{F}}

\newcommand{\Teich}{\mathcal{T}_{g,N+1}^{(1)}}
\newcommand{\T}{\mathrm{T}}
\newcommand{\corr}{\bf}
\newcommand{\vac}{|0\rangle}
\newcommand{\Ga}{\Gamma}
\newcommand{\new}{\bf}
\newcommand{\define}{\def}
\newcommand{\redefine}{\def}
\newcommand{\Cal}[1]{\mathcal{#1}}
\renewcommand{\frak}[1]{\mathfrak{{#1}}}
\newcommand{\Hom}{\rm{Hom}\,}
\newcommand{\refE}[1]{(\ref{E:#1})}
\newcommand{\refCh}[1]{Chapter~\ref{Ch:#1}}
\newcommand{\refS}[1]{Section~\ref{S:#1}}
\newcommand{\refSS}[1]{Section~\ref{SS:#1}}
\newcommand{\refT}[1]{Theorem~\ref{T:#1}}
\newcommand{\refO}[1]{Observation~\ref{O:#1}}
\newcommand{\refP}[1]{Proposition~\ref{P:#1}}
\newcommand{\refD}[1]{Definition~\ref{D:#1}}
\newcommand{\refC}[1]{Corollary~\ref{C:#1}}
\newcommand{\refL}[1]{Lemma~\ref{L:#1}}
\newcommand{\refEx}[1]{Example~\ref{Ex:#1}}
\newcommand{\R}{\ensuremath{\mathbb{R}}}
\newcommand{\C}{\ensuremath{\mathbb{C}}}
\newcommand{\N}{\ensuremath{\mathbb{N}}}
\newcommand{\Q}{\ensuremath{\mathbb{Q}}}
\renewcommand{\P}{\ensuremath{\mathcal{P}}}
\newcommand{\Z}{\ensuremath{\mathbb{Z}}}
\newcommand{\kv}{{k^{\vee}}}
\renewcommand{\l}{\lambda}
\newcommand{\gb}{\overline{\mathfrak{g}}}
\newcommand{\dt}{\tilde d}     
\newcommand{\hb}{\overline{\mathfrak{h}}}
\newcommand{\g}{\mathfrak{g}}
\newcommand{\h}{\mathfrak{h}}
\newcommand{\gh}{\widehat{\mathfrak{g}}}
\newcommand{\ghN}{\widehat{\mathfrak{g}_{(N)}}}
\newcommand{\gbN}{\overline{\mathfrak{g}_{(N)}}}
\newcommand{\tr}{\mathrm{tr}}
\newcommand{\gln}{\mathfrak{gl}(n)}
\newcommand{\son}{\mathfrak{so}(n)}
\newcommand{\spnn}{\mathfrak{sp}(2n)}
\newcommand{\sln}{\mathfrak{sl}}
\newcommand{\sn}{\mathfrak{s}}
\newcommand{\so}{\mathfrak{so}}
\newcommand{\spn}{\mathfrak{sp}}
\newcommand{\tsp}{\mathfrak{tsp}(2n)}
\newcommand{\gl}{\mathfrak{gl}}
\newcommand{\slnb}{{\overline{\mathfrak{sl}}}}
\newcommand{\snb}{{\overline{\mathfrak{s}}}}
\newcommand{\sob}{{\overline{\mathfrak{so}}}}
\newcommand{\spnb}{{\overline{\mathfrak{sp}}}}
\newcommand{\glb}{{\overline{\mathfrak{gl}}}}
\newcommand{\Hwft}{\mathcal{H}_{F,\tau}}
\newcommand{\Hwftm}{\mathcal{H}_{F,\tau}^{(m)}}

\newcommand{\car}{{\mathfrak{h}}}    
\newcommand{\bor}{{\mathfrak{b}}}    
\newcommand{\nil}{{\mathfrak{n}}}    
\newcommand{\vp}{{\varphi}}
\newcommand{\bh}{\widehat{\mathfrak{b}}}  
\newcommand{\bb}{\overline{\mathfrak{b}}}  
\newcommand{\Vh}{\widehat{\mathcal V}}
\newcommand{\KZ}{Kniz\-hnik-Zamo\-lod\-chi\-kov}
\newcommand{\TUY}{Tsuchia, Ueno  and Yamada}
\newcommand{\KN} {Kri\-che\-ver-Novi\-kov}
\newcommand{\pN}{\ensuremath{(P_1,P_2,\ldots,P_N)}}
\newcommand{\xN}{\ensuremath{(\xi_1,\xi_2,\ldots,\xi_N)}}
\newcommand{\lN}{\ensuremath{(\lambda_1,\lambda_2,\ldots,\lambda_N)}}
\newcommand{\iN}{\ensuremath{1,\ldots, N}}
\newcommand{\iNf}{\ensuremath{1,\ldots, N,\infty}}

\newcommand{\tb}{\tilde \beta}
\newcommand{\tk}{\tilde \varkappa}
\newcommand{\ka}{\kappa}
\renewcommand{\k}{\varkappa}
\newcommand{\ce}{{c}}

\newcommand{\Pif} {P_{\infty}}
\newcommand{\Pinf} {P_{\infty}}
\newcommand{\PN}{\ensuremath{\{P_1,P_2,\ldots,P_N\}}}
\newcommand{\PNi}{\ensuremath{\{P_1,P_2,\ldots,P_N,P_\infty\}}}
\newcommand{\Fln}[1][n]{F_{#1}^\lambda}
\newcommand{\tang}{\mathrm{T}}
\newcommand{\Kl}[1][\lambda]{\can^{#1}}
\newcommand{\A}{\mathcal{A}}
\newcommand{\U}{\mathcal{U}}
\newcommand{\V}{\mathcal{V}}
\newcommand{\W}{\mathcal{W}}
\renewcommand{\O}{\mathcal{O}}
\newcommand{\Ae}{\widehat{\mathcal{A}}}
\newcommand{\Ah}{\widehat{\mathcal{A}}}
\newcommand{\La}{\mathcal{L}}
\newcommand{\Le}{\widehat{\mathcal{L}}}
\newcommand{\Lh}{\widehat{\mathcal{L}}}
\newcommand{\eh}{\widehat{e}}
\newcommand{\Da}{\mathcal{D}}
\newcommand{\kndual}[2]{\langle #1,#2\rangle}
\newcommand{\cins}{\frac 1{2\pi\mathrm{i}}\int_{C_S}}
\newcommand{\cinsl}{\frac 1{24\pi\mathrm{i}}\int_{C_S}}
\newcommand{\cinc}[1]{\frac 1{2\pi\mathrm{i}}\int_{#1}}
\newcommand{\cintl}[1]{\frac 1{24\pi\mathrm{i}}\int_{#1 }}
\newcommand{\w}{\omega}
\newcommand{\ord}{\operatorname{ord}}
\newcommand{\res}{\operatorname{res}}
\newcommand{\nord}[1]{:\mkern-5mu{#1}\mkern-5mu:}
\newcommand{\codim}{\operatorname{codim}}
\newcommand{\ad}{\operatorname{ad}}
\newcommand{\Ad}{\operatorname{Ad}}
\newcommand{\supp}{\operatorname{supp}}

\newcommand{\Fn}[1][\lambda]{\mathcal{F}^{#1}}
\newcommand{\Fl}[1][\lambda]{\mathcal{F}^{#1}}
\renewcommand{\Re}{\mathrm{Re}}

\newcommand{\ha}{H^\alpha}

\define\ldot{\hskip 1pt.\hskip 1pt}
\define\ifft{\qquad\text{if and only if}\qquad}
\define\a{\alpha}
\redefine\d{\delta}
\define\w{\omega}
\define\ep{\epsilon}
\redefine\b{\beta} \redefine\t{\tau} \redefine\i{{\,\mathrm{i}}\,}
\define\ga{\gamma}
\define\cint #1{\frac 1{2\pi\i}\int_{C_{#1}}}
\define\cintta{\frac 1{2\pi\i}\int_{C_{\tau}}}
\define\cintt{\frac 1{2\pi\i}\oint_{C}}
\define\cinttp{\frac 1{2\pi\i}\int_{C_{\tau'}}}
\define\cinto{\frac 1{2\pi\i}\int_{C_{0}}}
\define\cinttt{\frac 1{24\pi\i}\int_C}
\define\cintd{\frac 1{(2\pi \i)^2}\iint\limits_{C_{\tau}\,C_{\tau'}}}
\define\dintd{\frac 1{(2\pi \i)^2}\iint\limits_{C\,C'}}
\define\cintdr{\frac 1{(2\pi \i)^3}\int_{C_{\tau}}\int_{C_{\tau'}}
\int_{C_{\tau''}}}
\define\im{\operatorname{Im}}
\define\re{\operatorname{Re}}
\define\res{\operatorname{res}}
\redefine\deg{\operatornamewithlimits{deg}}
\define\ord{\operatorname{ord}}
\define\rank{\operatorname{rank}}
\define\fpz{\frac {d }{dz}}
\define\dzl{\,{dz}^\l}
\define\pfz#1{\frac {d#1}{dz}}

\define\K{\Cal K}
\define\U{\Cal U}
\redefine\O{\Cal O}
\define\He{\text{\rm H}^1}
\redefine\H{{\mathrm{H}}}
\define\Ho{\text{\rm H}^0}
\define\A{\Cal A}
\define\Do{\Cal D^{1}}
\define\Dh{\widehat{\mathcal{D}}^{1}}
\redefine\L{\Cal L}
\newcommand{\ND}{\ensuremath{\mathcal{N}^D}}
\redefine\D{\Cal D^{1}}
\define\KN {Kri\-che\-ver-Novi\-kov}
\define\Pif {{P_{\infty}}}
\define\Uif {{U_{\infty}}}
\define\Uifs {{U_{\infty}^*}}
\define\KM {Kac-Moody}
\define\Fln{\Cal F^\lambda_n}
\define\gb{\overline{\mathfrak{ g}}}
\define\G{\overline{\mathfrak{ g}}}
\define\Gb{\overline{\mathfrak{ g}}}
\redefine\g{\mathfrak{ g}}
\define\Gh{\widehat{\mathfrak{ g}}}
\define\gh{\widehat{\mathfrak{ g}}}
\define\Ah{\widehat{\Cal A}}
\define\Lh{\widehat{\Cal L}}
\define\Ugh{\Cal U(\Gh)}
\define\Xh{\hat X}
\define\Tld{...}
\define\iN{i=1,\ldots,N}
\define\iNi{i=1,\ldots,N,\infty}
\define\pN{p=1,\ldots,N}
\define\pNi{p=1,\ldots,N,\infty}
\define\de{\delta}

\define\kndual#1#2{\langle #1,#2\rangle}
\define \nord #1{:\mkern-5mu{#1}\mkern-5mu:}
\newcommand{\MgN}{\mathcal{M}_{g,N}} 
\newcommand{\MgNeki}{\mathcal{M}_{g,N+1}^{(k,\infty)}} 
\newcommand{\MgNeei}{\mathcal{M}_{g,N+1}^{(1,\infty)}} 
\newcommand{\MgNekp}{\mathcal{M}_{g,N+1}^{(k,p)}} 
\newcommand{\MgNkp}{\mathcal{M}_{g,N}^{(k,p)}} 
\newcommand{\MgNk}{\mathcal{M}_{g,N}^{(k)}} 
\newcommand{\MgNekpp}{\mathcal{M}_{g,N+1}^{(k,p')}} 
\newcommand{\MgNekkpp}{\mathcal{M}_{g,N+1}^{(k',p')}} 
\newcommand{\MgNezp}{\mathcal{M}_{g,N+1}^{(0,p)}} 
\newcommand{\MgNeep}{\mathcal{M}_{g,N+1}^{(1,p)}} 
\newcommand{\MgNeee}{\mathcal{M}_{g,N+1}^{(1,1)}} 
\newcommand{\MgNeez}{\mathcal{M}_{g,N+1}^{(1,0)}} 
\newcommand{\MgNezz}{\mathcal{M}_{g,N+1}^{(0,0)}} 
\newcommand{\MgNi}{\mathcal{M}_{g,N}^{\infty}} 
\newcommand{\MgNe}{\mathcal{M}_{g,N+1}} 
\newcommand{\MgNep}{\mathcal{M}_{g,N+1}^{(1)}} 
\newcommand{\MgNp}{\mathcal{M}_{g,N}^{(1)}} 
\newcommand{\Mgep}{\mathcal{M}_{g,1}^{(p)}} 
\newcommand{\MegN}{\mathcal{M}_{g,N+1}^{(1)}} 

\define \sinf{{\widehat{\sigma}}_\infty}
\define\Wt{\widetilde{W}}
\define\St{\widetilde{S}}
\newcommand{\SigmaT}{\widetilde{\Sigma}}
\newcommand{\hT}{\widetilde{\frak h}}
\define\Wn{W^{(1)}}
\define\Wtn{\widetilde{W}^{(1)}}
\define\btn{\tilde b^{(1)}}
\define\bt{\tilde b}
\define\bn{b^{(1)}}
\define \ainf{{\frak a}_\infty} 

%
\define\eps{\varepsilon}    
\newcommand{\e}{\varepsilon}
\define\doint{({\frac 1{2\pi\i}})^2\oint\limits _{C_0}
       \oint\limits _{C_0}}                            
\define\noint{ {\frac 1{2\pi\i}} \oint}   
\define \fh{{\frak h}}     
\define \fg{{\frak g}}     
\define \GKN{{\Cal G}}   
\define \gaff{{\hat\frak g}}   
\define\V{\Cal V}
\define \ms{{\Cal M}_{g,N}} 
\define \mse{{\Cal M}_{g,N+1}} 
\define \tOmega{\Tilde\Omega}
\define \tw{\Tilde\omega}
\define \hw{\hat\omega}
\define \s{\sigma}
\define \car{{\frak h}}    
\define \bor{{\frak b}}    
\define \nil{{\frak n}}    
\define \vp{{\varphi}}
\define\bh{\widehat{\frak b}}  
\define\bb{\overline{\frak b}}  
\define\KZ{Knizhnik-Zamolodchikov}
\define\ai{{\alpha(i)}}
\define\ak{{\alpha(k)}}
\define\aj{{\alpha(j)}}
\newcommand{\calF}{{\mathcal F}}
\newcommand{\ferm}{{\mathcal F}^{\infty /2}}
\newcommand{\Aut}{\operatorname{Aut}}
\newcommand{\End}{\operatorname{End}}
\newcommand{\red}{\color[rgb]{1,0,0}}
\newcommand{\blue}{\color[rgb]{0,0,1}}
\newcommand{\viol}{\color[rgb]{1,0,1}}

\newcommand{\laxgl}{\overline{\mathfrak{gl}}}
\newcommand{\laxsl}{\overline{\mathfrak{sl}}}
\newcommand{\laxso}{\overline{\mathfrak{so}}}
\newcommand{\laxsp}{\overline{\mathfrak{sp}}}
\newcommand{\laxs}{\overline{\mathfrak{s}}}
\newcommand{\laxg}{\overline{\frak g}}
\newcommand{\bgl}{\laxgl(n)}
\newcommand{\tX}{\widetilde{X}}
\newcommand{\tY}{\widetilde{Y}}
\newcommand{\tZ}{\widetilde{Z}}


\title[]{Inversion of the Abel--Prym map in presence of an additional involution}
\author[O.K.Sheinman]{O.K.Sheinman}
\address{Steklov Mathematical Institute of the Russian Academy of Sciences}
\dedicatory{}
\maketitle
\begin{abstract}
Unlike Abel map of the symmetric power of a Riemann surface onto its Jacobian, the Abel--Prym map generically can not be reversed by means of conventional technique related to the Jacobi inversion problem, and of its main ingredient, namely the Riemann vanishing theorem. It happens because the corresponding analog of the Riemann vanishing theorem gives twice as many points as the dimension of the Prym variety. However, if the Riemann surface has a second involution commuting with the one defining the Prym variety and satisfying a certain additional condition, an analog of the Jacobi inversion can be defined, and expressed in terms of the Prym theta function. We formulate these conditions and refer to the pairs of involutions satisfying them as to pairs of the first type. We formulate necessary conditions for the pair of involutions to be a pair of the first type, and give a series of examples of curves with such pairs of involutions, mainly spectral curves of Hitchin systems, and also a spectral curve of the Kovalewski system.

Bibliography: 14 titles.

\end{abstract}
\tableofcontents
\section{Introduction}
Let $\A:{\rm Sym}^g\Sigma\to Jac(\Sigma)$ be the Abel map of the $g$th symmetric degree of a genus $g$ Riemann surface $\Sigma$ onto its Jacobian.  It is wellknown that $\A$ is invertible almost everywhere, the problem of its reversion is named after Jacobi.
By a classical Riemann theorem, the preimage of a point $\phi\in Jac(\Sigma)$ is given by the set of zeroes of the function $F(P)=\theta(A(P)-\phi+\Delta)$ where $\theta$ is the Riemann theta function, $P\in\Sigma$, $\Delta\in Jac(\Sigma)$ does not depend on $\phi$.  In \cite{Dubr_theta}, according to ideas going back to Riemann, the Jacobi inversion problem is set and resolved as the problem of finding explicit expressions for symmetric functions of zeroes of $F(P)$ in terms of theta functions on the Jacobian.

The situation is different if we consider maps of symmetric powers of $\Sigma$ to other Abelian varieties, such as its Prym variety and its coverings. Let $\Sigma$ to possess a holomorphic involution $\tau$. Prym differentials are defined as holomorphic 1-forms on $\Sigma$ which are skew-symmetric with respect to this involution. Below, in \refS{ogr}, the normalized Prym differentials are defined. The analog of the Abel transform with only normalized Prym differentials involved, is referred to as the Abel--Prym transform, and its image  as the Prym variety (Prymian). Let $\A$ stay for the Abel--Prym map as well, and $Prym_\tau(\Sigma)$ stay for the Prymian of the pair $\Sigma$, $\tau$. Let dimension of the space of Prym differentials to be equal to $h$, $\Pi$ be a $h\times h$ symmetric matrix with negative defined real part constructed from periods and half-periods of the normalized Prym differentials (see \refS{ogr}), $E$ be the unit matrix, $\Z(2\pi i E,\Pi)$ be a lattice generated by columns of the matrices $2\pi i E,\Pi$. Then the Abelian variety $\C^h/\Z(2\pi i E,\Pi)$ is isogenic to the Prym variety of the pair $\Sigma,\tau$ \cite{Fay}, and is equal to the last in the two particular cases, namely if the involution $\tau$ has only two fixed points, or no fix point. We refer to that variety as isoPrymian, and denote it $isoPrym_\tau(\Sigma)$. The theta function with the Riemann matrix $\Pi$ is referred to as the Prym $\theta$-function (see \pageref{thetaprym} for the definition).

The Abel--Prym map is well-defined as the mapping of $\Sigma$ to $isoPrym_\tau(\Sigma)$. As such, it can be continued to ${\rm Sym}^g\Sigma$, and to ${\rm Sym}^h\Sigma$ as well, but generically in none of the cases any analog of Jacobi inversion exists, because the function $F(P)$ constructed with help of the Prym $\theta$-function, has $2h$ zeroes in $\Sigma$ \cite[corollary 5.6]{Fay} (see also \refL{num_zero} below). In the case when the branching number of the covering $\Sigma\to\Sigma/\tau$ is equal to 2, and respectively $g=2h$, the Riemann vanishing theorem  gives a map $isoPrym_\tau\to {\rm Sym}^g\Sigma$, but this is the map of an $h$-dimensional variety to a $2h$-dimensional one, that is by no means a reversion. In the present work we consider the case when  an analog of the Jacobi inversion can be defined nevertheless.

Assume, the curve admits two nontrivial commuting holomorphic involutions: $\tau_1$ and $\tau_2$, and we consider the Abel--Prym transform with respect to $\tau_1$. Then the two fundamentally different cases can occur: in the first case all Prym differentials (with respect to $\tau_1$) are invariant with respect to $\tau_2$, while in the second case -- not. An example of the first option is given by the spectral curve of the Hitchin system with the structure group $SO(4)$ on a genus 2 Riemann surface (see \refEx{SO4_2}, \refS{Examples}).

The subject of this article are the curves with pairs of involutions of the first type. The main result is as follows.
\begin{theorem}\label{T:birat}
Let $\Sigma,\tau_1,\tau_2$ be a curve with given pair of commuting holomorphic involutions of the first type, $Prym_i(\Sigma)$ be its Prymian with respect to the involution $\tau_i$, $h_i=\dim Prym_i(\Sigma)$, $\Sigma_i=\Sigma/\tau_i$, $2n_i$ be the branching number of $\Sigma$ over $\Sigma_i$, $i=1,2$. Then there are only four options for the pair $(2n_1,2n_2)$, namely $(2,2)$, $(4,0)$, $(0,4)$, $(0,0)$. In all cases except for $(4,0)$
\[
  Prym_1(\Sigma)\simeq\rm{Sym}^{h_1} \Sigma_2.
\]
In the case of $(4,0)$
\[
  Prym_1(\Sigma)\simeq(\rm{Sym}^{h_1} \Sigma_2)/\Z_2,
\]
where the symbol $\simeq$ stays for birational equivalence.
\end{theorem}
It is immediately implied by the following theorem.
\begin{theorem}\label{T:bihol}
Under conditions of \refT{birat}, let $\Sigma_2^\prime$ be interior of the polygone obtained by dissecting of $\Sigma_2$ along the chosen fundamental cycles. Then the Abel--Prym map $\A : \Sigma\to Prym_1(\Sigma)$ can be pushed down onto $\Sigma_2$ and gives a biholomorphic equivalence between $\rm{Sym}^{h_1} (\Sigma_2^\prime)$ and some open dense subset in $Prym_1(\Sigma)$ if the pair of branching numbers is not equal to $(4,0)$. If it is equal to $(4,0)$ then $\A$ gives a biholomorphic equivalence between  $\rm{Sym}^{h_1} (\Sigma_2^\prime)$ and a 2-fold covering of $Prym_1(\Sigma)$.
\end{theorem}
Then  \refT{birat} is implied by the following. For complex varieties a biholomorphic equivalence of their open dense subsets, and analiticity of the closure of its graph, imply their bimeromorphic equivalence. If, moreover, the varieties are projective, the  bimeromorphic equivalence implies their birational equivalence (see \refS{ogr} for details).

In turn, proof of the Theorem \ref{T:bihol} relies on the following analog of the Riemann vanishing theorem for the theta function. As above, let $\Sigma'$ stay for the polygon obtained by dissection of $\Sigma$ along fundamental cycles.
\begin{proposition}\label{P:atR}
Let $\theta$ be the Prym theta function on the universal covering $\C^{h_1}$ of $Prym_1(\Sigma)$, $\phi\in\C^{h_1}$. Then the function
$ F(P)=\theta(\A(P)-\phi),\quad P\in\Sigma'$
is $\tau_2$-invariant, and for almost all $\phi$ has $h_1$ $\tau_2$-invariant pairs of zeroes, well defined as points of  $\Sigma$, and depending on the image of $\phi$ in $Prym_1(\Sigma)$ only. The image $\widetilde\phi$ of the zero divisor under the Abel--Prym map is related with $\phi$ by a constant affine transformation whose form is made more precise in \refL{inver} below.
\end{proposition}
Proposition \ref{P:atR} is a summary of lemmas \ref{L:num_zero} and \ref{L:inver}. The above formulated statements are proved in \refS{ogr} of the present work.

The Riemann vanishing theorem and its above formulated analog given by Proposition \ref{P:atR} provide an implicit solution to the inversion problem. For applications, it is important to explicitly compute zeroes of the function $F(P)$. In the case of hyperelliptic curves (whose Jacobians coincide with their Prymians with respect to the hyperelliptic involution) the preimage of a point of Jacobian under the Abel transfom is given by zeroes of a polynomial whose coefficients can be explicitly expressed in terms of $\wp$-functions of the hyperelliptic curve \cite{BEL}. For more general curves, an approach has been proposed,  going back to Riemann, enabling one to compute symmetric functions of zeroes of $F(P)$ in terms of the Riemann theta function of the curve \cite{Dubr_theta}. In \refS{theta-form} of the present work we generalize that approach onto the case of curves with a pairs of involutions of the first type, and compute symmetric functions of zeroes in terms of the Prym theta function of the curve.

Remainder of the paper is devoted to applications of the above described  technique to finding the trajectories of integrable systems, mainly of Hitchin systems and their degenerations, but also of the Kovalewski system. We use a classical idea of the theory of integrable systems: to map the straight windings of invariant tori of the system (of isoPrymians in our case) to the phase space with original coordinates by means the Jacobi inverse transform (or of its analog constructed here).

In the auxiliary \refS{Hitch}, following \cite{Sh_FAN_2019,BorSh}, we introduce Hitchin systems in frame of the method of Separation of Variables, by giving their spectral curves and Poisson brackets.

In \refS{Examples}, we address the Hitchin systems with structure groups $SL(2)$, $SO(4)$, $Sp(4)$ on genus 2 and 3 Riemann surfaces, and their degenerations, as well as the Kovalewski system, and show that their spectral curves possess a pair of involutions of the type 1. Hence the listed systems are explicitly resolvable in Prym theta functions by means of the above described methods. Besides, we find out that these curves provide examples of all types of branching listed in \refT{birat}. For Hitchin systems with the structure groups  $SL(2)$, $SO(4)$, on genus 2 curves, a general solution has been obtained in \cite{Sh_SO4} by that method, for the first time as for $SO(4)$ .

The author thanks V.V.Shokurov and A.V.Fonarev for instructive discussions.

\section{Curves with a pair of involutions of the type 1, and reversion theorem}\label{S:ogr}
In this section the curves possessing a pair of involutions of the type 1 are classified, and a reversion theorem for the Abel--Prym map has been proved for them.

As above, $g_1$, $h_1$ stay for the genus of $\Sigma_1$ and for the dimension of $Prym_1(\Sigma)$, respectively,  $g_2$, $h_2$ are the same for $\Sigma_2$ and $Prym_2(\Sigma)$, $\widehat g$ is a genus of $\Sigma$, $2n_1$, $2n_2$ are branch numbers (degrees of the branch divisors) of the corresponding coverings.
\begin{lemma}\label{L:k1t}
Let $\Sigma,\tau_1,\tau_2$ be a curve with a pair of commuting involutions of the first type. Then the following four options can occur:
\begin{itemize}
  \item[$1^\circ$.]
   $n_1=n_2=1$: then each involution has two fixed points (the branch points of the corresponding coverings), $g_1=g_2$ (let $\widetilde{g}$ stay for their common value), the genus $\widehat{g}$ of $\Sigma$ is even, and  $\widehat{g}=2\widetilde{g}$, $h_1=h_2=\widetilde{g}$.
  \item[$2^\circ$.]
  $n_1=2$, $n_2=0$: in this case the second covering is unramified, a genus of $\Sigma$ is odd, $g_1=g_2-1$, $h_1=g_1+1$, $h_2=g_2-1$. In particular, $h_1>h_2$.
  \item[$3^\circ$.]
  $n_1=0$, $n_2=2$: in this case the first covering is unramified, a genus of $\Sigma$ is odd, but $g_1=g_2+1$, $h_1=g_1-1$, $h_2=g_2+1$. In particular, $h_1<h_2$.
  \item[$4^\circ$.]
  $n_1=n_2=0$: in this case $g_1=g_2$ (let $\widetilde{g}$ stay for their common value), a genus of $\Sigma$ is odd, and equal to $\widehat{g}=2\widetilde{g}-1$, $h_1=h_2=\widetilde{g}-1$.
\end{itemize}
\end{lemma}
\begin{proof}
By definition of a pair of the first type, $\tau_2$ has at least $h_1$ independent symmetric differentials, hence the number of the skew-symmetric ones is not bigger than $\widehat g-h_1$, i.e. $h_2\le\widehat g-h_1$, or $h_1+h_2\le\widehat g$. We have $h_i=g_i+n_i-1$, $i=1,2$ (see \cite[p. 85]{Fay}). It follows that $g_1+g_2+n_1+n_2-2\le\widehat g$. Plug here, first $\widehat g=2g_1+n_1-1$ (this is nothing but the Riemann--Hurwitz fomula, cf. \cite{Fay}, right there), and then $\widehat g=2g_2+n_2-1$. In the first case we obtain $g_2+n_2-1\le g_1$, while in the second case $g_1\le g_2-n_1+1$. It follows that
\[
     n_1+n_2\le 2.
\]
Moreover, the number $n_1+n_2$ must be even. Indeed,  $\widehat{g}=2g_1+n_1-1$ and  $\widehat{g}=2g_2+n_2-1$ imply that $2\widehat{g}=2(g_1+g_2)+n_1+n_2-2$. Besides $n_1+n_2$, all summands are even in the last equality.

Thus, either $n_1+n_2=2$, or $n_1+n_2=0$. The first option takes place in the cases $1^\circ - 3^\circ$ of the Lemma, while the second in the case $4^\circ$.

In the first case,  $h_i=g_i+n_i-1$, $i=1,2$, and $n_1=n_2=1$ imply $h_1=g_1$, $h_2=g_2$. The $\widehat g=2g_i+n_i-1$, $i=1,2$ impliy $\widehat g=2g_i$, $i=1,2$, in particular $g_1=g_2$.

In the second case we similarly have $h_1=g_1+1$, $h_2=g_2-1$. Then, by $g_2+n_2-1\le g_1$, we obtain $g_2-1\le g_1$, and by $g_1\le g_2-n_1+1$ it follows $g_1\le g_2-1$, and finally $g_1=g_2-1$.

The third case is being considered in a similar way to the second case.

In the case $4^\circ$ we have $\widehat{g}=2g_1-1$ and  $\widehat{g}=2g_2-1$ which imply $g_1=g_2$. Hence $h_i=g_i+n_i-1=g_i-1$, $i=1,2$.
\end{proof}
Examples of curves with a pair of involutions of the first type are given in \refS{Examples}. Here we shall obtain certain consequences of \refL{k1t}, and prove the main theorem.
\begin{corollary}\label{C:h1=g2}
In the cases $1^\circ$--$3^\circ$ of \refL{k1t} $h_1=g_2$, $h_2=g_1$.
\end{corollary}
\begin{remark}\label{R:on_Th1.1}
For the first glance, by \refC{h1=g2}, in the cases $1^\circ$, $3^\circ$ of \refL{k1t}  the \refT{birat} does not require any special proof. Indeed, in those cases it descends to the statement that the Prymian with respect to the involution $\tau_1$ is birationally equivalent to the Jacobian of $\Sigma_2$. It seems to be obvious, since the Prym differentials with respect to $\tau_1$ are invariant with respect to $\tau_2$, hence they push down onto $\Sigma_2$, giving holomorphic differentials there. Their number is equal to the genus of  $\Sigma_2$ (because $h_1=g_2$), hence they form a base of holomorphic differentials on $\Sigma_2$. For this reason the Abel--Prym map for the curve $\Sigma$ is nothing but the Abel map for $\Sigma_2$. However, for a proof of the biholomorphy of the map $Prym_1(\Sigma)\to Jac(\Sigma_2)$, and of their birational equivalence (which would follow by projectivity of the varieties) we are missing a necessary information on behaviour of the fundamental cycles under the projection $\Sigma\to\Sigma_2$ in this line of arguing.

In the case $4^\circ$ of the Lemma the relations $h_1=h_2=g_1-1=g_2-1$ hold, hence there is no equivalence between the Prymian of the spectral curve with respect to $\tau_1$, and  the Jacobian of  $\Sigma_2$. Below, we give an independent proof of  \refT{birat}, valid in all cases.
\end{remark}
First of all, we establish the correspondence between all normalized holomorphic differen\-tials and normalized Prym differentials on $\Sigma$, and define the Riemann matrix of the Prym variety (called Prym matrix). To be specific, we do it for the involution $\tau_1$. According to \cite{Fay}, there exists a base of cycles $a_i,b_i$ ($i=1,\ldots, g_1$), $a_i,b_i$ ($i=g_1+1,\ldots, h_1=g_1+n_1-1$), $a_{i+h_1},b_{i+h_1}$ ($i=1,\ldots, g_1$) on $\Sigma$, where the first and the second groups of cycles are pulled back from $\Sigma_1$,  $\pi(a_i)=\pi(\a_{i+h_1})$, $\pi(b_i)=\pi(b_{i+h_1})$, and the following relations hold:
\begin{equation}\label{E:cycles}
\begin{aligned}
    &\tau_1(a_i)+a_{i+h_1}= \tau_1(b_i)+b_{i+h_1}=0,\ i=1,\ldots g_1\\
    &\tau_1(a_i)+a_i= \tau_1(b_i)+b_i=0,\ i=g_1+1,\ldots,g_1+n_1-1=h_1.
\end{aligned}
\end{equation}
Let $\{ w_i|  i=1,\ldots, \widehat{g} \}$ be a dual base of normalized holomorphic differentials. For any differential $w$ let $\tau_1^*w$ be a differential obtained by change of veriables $\tau_1$ in $w$: $\tau_1^*w(P)=w(\tau_1 P)$. Then $w_{i+h_1}=-\tau_1^*w_i$ ($i=1,\ldots g_1$), $\tau_1^*w_i=-w_i$ ($i=g_1+1,\ldots,h_1$). Differentials $\{\w_i=w_i+w_{i+h_1} | i =1,\ldots g_1\}$ and $\{ \w_i=w_i|i=g_1+1,\ldots,h_1\}$ form a base of Prym differentials on $\Sigma$. This base is normalized in a sense that $\oint_{a_j}\w_k=2\pi i\d_{jk}$, $i,j=1,\ldots,h_1$. The Riemann matrix of the variety $isoPrym_1$ is the matrix $\Pi=(\Pi_{ij})_{i,j=1,\ldots,h_1}$ where
\begin{equation}\label{E:Rim_matr}
 \Pi_{ij}=\oint_{b_j}\w_i\ (j=1,\ldots,g_1);\quad
 \Pi_{ij}=\frac{1}{2}\oint_{b_j}\w_i\ (j=g_1+1,\ldots,h_1)
\end{equation}
(cf. \cite[Eq. (92)]{Fay}). The theta function $\theta(z,\Pi)=\sum_{N\in \Z^g}\exp(\frac{1}{2}(\Pi N,N)+(z,N))$\label{thetaprym}  is referred to as Prym theta function, $z=(z_1,\ldots,z_{h_1})$. The lattice  $\Z({2\pi i}E,\Pi)\subset\C^{h_1}$ generated by the columns of $h_1\times h_1$ matrices $2\pi iE$ and $\Pi$ is referred to as the period lattice, $isoPrym_1=\C^{h_1}/\Z({2\pi i}E,\Pi)$. The map $\A : \Sigma\to isoPrym_1$:
\[
  \A(\ga)= \left(\int_{\ga_0}^{\ga}\overline{\w}\right)\,({\rm mod}\,\Z({2\pi i}E,\Pi))
\]
where $\overline{\w}=(\w_1,\ldots,\w_{h_1})^T$, $\tau_1\ga_0=\ga_0$ is reffered to as the Abel--Prym map. Below, we suppress the indication on $\Pi$ in the notation of the theta function.

Let $F(P)=\theta(\A(P)-e)$ ($P\in\Sigma$, $e\in\C^{h_1}$).
\begin{lemma}[\cite{Fay}, Corollary 5.6]\label{L:num_zero}
If  $F(P)$ does not identically vanish then it has exactly $2h_1$ zeroes (counted with their multiplicity) on $\Sigma$.
\end{lemma}
\begin{proof}
We give here a proof in spirit of \cite[Lemma 2.4.1]{Dubr_theta}, with addition of two remarks following from the relations between the cycles, between the base differentials, and from the form of the Prym matrix, and  implying the relations \refE{sdvig2}, \refE{sdvig3} below. The relations obtained in the meanwhile will be helpful below.

Let $\Sigma'$ be the domain obtained by dissection of the Riemann surface $\Sigma$ along its basis cycles. Since $F(P)$ is holomorphic in $\Sigma'$, the number of its zeroes is equal to
\begin{equation}\label{E:num_zero}
 \sum_{P\in \Sigma'}\res_Pd\ln F(P) = \frac{1}{2\pi i}\oint\limits_{\partial\Sigma'}d\ln F(P).
\end{equation}
Let $F^+(P)$ be a value of the function $F$ at the image of the point $P$ on the segment $a_k$ (or $b_k$) of the boundary of $\Sigma'$, and $F^-(P)$ be the same on the segment $a_k^{-1}$ (or $b_k^{-1}$) (these values are known as values of the function on "different cut banks"). Then
\begin{equation}\label{E:gran}
 \frac{1}{2\pi i}\oint\limits_{\partial\Sigma'}d\ln F(P)=\frac{1}{2\pi i} \sum_{k=1}^{\widehat{g}} \left(\oint\limits_{a_k}+\oint\limits_{b_k}\right) (d\ln F^+-d\ln F^-).
\end{equation}
We will use the notation $\A^\pm_j(P)$ in the same sense, where $\A_j(P)=\int_{\ga_0}^P\w_j$. If $P$ is a point on $a_k$ then
\begin{equation}\label{E:sdvig0}
  \A_j^-(P)=\A_j^+(P)+\oint_{b_k}\w_j,\ j=1,\ldots,h_1;\ k=1,\ldots,\widehat{g},
\end{equation}
because the way from $a_k$ to $a_k^{-1}$ runs along the $b$-cycle.

For $k=1,\ldots,g_1$, the relation \refE{sdvig0} gives $\A_j^-(P)=\A_j^+(P)+\Pi_{jk}$. From the transformation low for $\theta$-functions, we have $\ln F^-(P) - \ln F^+(P)=-\frac{1}{2}\Pi_{kk}-\A_k(P)+e_k$, which implies
\begin{equation}\label{E:sdvig1}
   d\ln F^+(P) - d\ln F^-(P) = \w_k(P),\ P\in a_k,\ k=1,\ldots,g_1.
\end{equation}
By invariance of the integral with respect to a change of variables, we have
\[
 \oint_{b_k}\w_j=\oint_{\tau_1(b_k)}\tau_1^*\w_j=\oint_{-b_{k+h_1}}(-\w_j)=\oint_{b_{k+h_1}}\w_j
\]
for $k=1,\ldots,g_1$. For this reason, if $k=1,\ldots,g_1$ then we have $\A_j^-(P)=\A_j^+(P)+\Pi_{jk}$ also for $P\in a_{k+h_1}$, which implies
\begin{equation}\label{E:sdvig2}
  d\ln F^+(P) - d\ln F^-(P) = \w_k(P),\ P\in a_{k+h_1},\ k=1,\ldots,g_1
\end{equation}
for any $j=1,\ldots,h_1$.

Similarly, for $k=g_1+1,\ldots,h_1$, due to the coefficient $1/2$ in \refE{Rim_matr} we have $\A_j^-(P)=\A_j^+(P)+2\Pi_{jk}$, which implies
\begin{equation}\label{E:sdvig3}
 d\ln F^+(P) - d\ln F^-(P) = 2\w_k(P),\ P\in a_k,\ k=g_1+1,\ldots,h_1.
\end{equation}
For $P\in b_k$ we have $\A_j^+(P)-\A_j^-(P)=2\pi i\d_{jk}$, and the transformation low for $\theta$-functions gives $F^+(P)-F^-(P)=0$, $d\ln F^+(P)-d\ln F^-(P)=0$. Hence a contribution of $b$-cycles into the sum \refE{gran} is equal to zero. Comparing \refE{num_zero}, \refE{gran}, \refE{sdvig1}, \refE{sdvig2} and \refE{sdvig3}, we obtain
\[
 \sum_{P\in \Sigma'}\res_Pd\ln F(P) = \frac{1}{2\pi i} \sum_{k=1}^{h_1} \oint\limits_{a_k} 2\w_k = 2h_1.
\]
\end{proof}
\begin{lemma}\label{L:inver}
If $F(P)$ does not identically vanish, $P_1,\ldots,P_{2h_1}$ are its zeroes on $\Sigma'$  then $\A(P_1+\ldots+P_{2h_1})=\tilde{e}+\Delta$ where $\tilde{e}_j=e_j$ for $j=1,\ldots,g_1$,  $\tilde{e}_j=2e_j$ for $j=g_1+1,\ldots,h_1$, and $\Delta$ does not depend on $e$.
\end{lemma}
\begin{proof}
We set $\zeta=\A(P_1+\ldots+P_{2h_1})$. In analogy with \refE{num_zero},
\[
  \zeta_j=\sum_{k=1}^{2h_1}\res_{P_k}\A_j(P)d\ln F(P) = \frac{1}{2\pi i}\oint\limits_{\partial\Sigma'}\A_j(P)d\ln F(P),
\]
where $j=1,\ldots,h_1$, and further on
\begin{equation*}
\zeta_j =\frac{1}{2\pi i} \sum_{k=1}^{\widehat{g}} \left(\oint\limits_{a_k}+\oint\limits_{b_k}\right) (A_j^+d\ln F^+ - A_j^-d\ln F^-).
\end{equation*}
Then we use the relations \refE{sdvig0}--\refE{sdvig3} written in the form
\[
    \oint\limits_{b_k}\w_j = \epsilon_k\Pi_{jk};\quad
    d\ln F^+-d\ln F^- = \left\{
                          \begin{array}{ll}
                            \epsilon_k\w_k & \hbox{on $a_k$;} \\
                            0 & \hbox{on $b_k$;}
                          \end{array}
                        \right.\quad
  A_j^- - A_j^+ = \left\{
                          \begin{array}{ll}
                            \epsilon_k\Pi_{jk} & \hbox{on $a_k$;} \\
                            2\pi\d_{jk} & \hbox{on $b_k$,}
                          \end{array}
                        \right.
\]
where we keep the following convention: $\Pi_{jk}=\Pi_{j,k-h_1}$, $\w_k=\w_{k-h_1}$, $\epsilon_k=\epsilon_{k-h_1}$ for $k=h_1+1,\ldots, \widehat{g}$. Then
\begin{equation}\label{E:gran_1}
\begin{aligned}
 \zeta_j &= \frac{1}{2\pi i} \sum_{k=1}^{\widehat{g}} \oint\limits_{a_k} \left( A_j^+d\ln F^+ - (A_j^+ + \epsilon_k\Pi_{jk})(d\ln F^+ -\epsilon_k\w_k) \right) +     \\
 &\phantom{aaa} \frac{1}{2\pi i} \sum_{k=1}^{\widehat{g}} \oint\limits_{b_k} (A_j^+d\ln F^+ - (A_j^+ + 2\pi i\d_{jk})d\ln F^+ )=\\
 &= \frac{1}{2\pi i} \sum_{k=1}^{\widehat{g}}\left( \oint\limits_{a_k} \epsilon_k A_j^+\w_k - \epsilon_k\Pi_{jk}\oint_{a_k}d\ln F^+ + 2\pi i\epsilon_k \Pi_{jk} \right) + \oint_{b_j}d\ln F^+,
\end{aligned}
\end{equation}
where $\epsilon_k=1$ for $k=1,\ldots,g_1,h_1,\ldots,\widehat{g}$, and $\epsilon_k=2$ for $k=g_1+1,\ldots,h_1$ (we used relations \refE{sdvig0}--\refE{sdvig3} in course of the computation). Evaluations of $F^+$ at the ends of the segment $a_k$ are equal, and $\ln F^+$ is defined up to addition of a multiple of $2\pi i$, for this reason $\oint_{a_k}d\ln F^+ = 2\pi in_k$ where $n_k\in\Z$. The summands in brackets in the last line of \refE{gran_1} are independent of $e$, they contribute in $\Delta_j$, moreover the last two ones of them are $j$th coordinates of some elements of the period lattice.

Let $Q_j$ and $\widetilde{Q}_j$ be the beginning, and the end of the segment $b_j$, respectively. Then
\begin{equation}\label{E:shift_12}
\begin{aligned}
   \oint_{b_j}d\ln F^+ &= \ln F^+(\widetilde{Q}_j) - \ln F^+(Q_j)+2\pi i m_j = \\
   &= \ln\theta(A(Q_j)-e+\epsilon_jf_j) - \ln\theta(A(Q_j)-e) + 2\pi i m_j,
\end{aligned}
\end{equation}
where $f_j$ is the $j$th column of the matrix $\Pi$, $\epsilon_j$ has the same meaning as above. Notice that a difference of the arguments of the  $\theta$-function is equal here to the period of the lattice if $j=1,\ldots,g_1$, and to the double period for $j=g_1+1,\ldots,h_1$. By the transformation low for $\theta$-functions we obtain
\begin{equation}\label{E:shift_13}
 \oint_{b_j}d\ln F^+ = \epsilon_je_j - \frac{1}{2}\epsilon_j\Pi_{jj} - \epsilon_jA_j^+(Q_j)+2\pi im_j.
\end{equation}
Observe also that $\epsilon_je_j=\tilde{e}_j$ for $j=1,\ldots,h_1$. Hence, up to elements of the period lattice, we obtain $\zeta=\tilde{e}+\Delta$ where
\[
  \Delta_j= \frac{1}{2\pi i} \sum_{k=1}^{\widehat{g}} \oint\limits_{a_k}\epsilon_kA_j^+\w_k - \frac{1}{2}\epsilon_j\Pi_{jj} - \epsilon_jA_j^+(Q_j).
\]
\end{proof}
Observe that the involution $\tau_2$ was by no means involved in Lemmas \ref{L:inver}, \ref{L:num_zero} and their proofs.
\begin{proof}[Proof of the theorem \ref{T:bihol}]
The involution $\tau_2$ is nontrivial, hence the set of its fixed points is either finite or empty. Its completion is an open dense subset in $\Sigma$. It follows from \refL{num_zero} and \refL{inver} that $\A$ establishes a biholomorphic equivalence between the set of all non-ordered sets of $h_1$ $\tau_2$-invariant pairs of points of that subset, and an open dense subset of a certain covering of the Prymian, trivial for all pairs of branch numbers except for $(4,0)$. For the last, the  covering is two-fold, and the covering map is nothing but a projection onto a quotient by $\Z_2$. It basically coincides with the statement of \refT{bihol} because the unordered sets of $h_1$ $\tau_2$-invariant pairs of points of the Riemann surface  $\Sigma$ are nothing but points of the variety ${\rm Sym}^{h_1}\Sigma_2$.
\end{proof}
As it was noticed in the Introduction, \refT{birat} immediately follows from Theorem \ref{T:bihol}. Indeed, it has been already proven that the equivalence maps in \refT{birat} are biholomorphic on some open dense subsets. We still need to prove that closures of graphs of those maps are analitic, in the sense that they are sets of zeroes of holomorphic functions. The last is obvious, since by \refL{inver} (essentially, by the Riemann theorem) such function is given by $F(P)=\theta(\A(P)-\epsilon^{-1}\phi+\epsilon^{-1}\Delta)$, and by its continuation onto ${\rm Sym}^{h_1}\Sigma_2$.

In applications to integrable systems, in the cases $1^\circ$, $3^\circ$ one can use the Jacobi inversion, because the trajectories linearize on the Jacobian of $\Sigma_2$.

\section{$\theta$-functional formula for symmetric functions of zeroes} \label{S:theta-form}

Here we address the problem of effective reversion of the Abel--Prym map in the case of two involutions of the first type. The solution proposed here relies on the fact that the transform $\A(P)$, as well as the function $F(P)=\theta(\A(P)-\phi)$, and the set of its zeroes, are invariant with respect to the involution $\tau_2$, for any $\phi$.
\begin{lemma}\label{L:tau2inv}
The transform $\A$ is invariant with respect to the involution $\tau_2$.
\end{lemma}
\begin{proof}
By definition,
\[
  \A(\tau_2P)=\int\limits_{Q_0}^{\tau_2P}\w = \int\limits_{Q_0}^{P}\w + \int\limits_{P}^{\tau_2P}\w = \A(P)+ \int\limits_{P}^{\tau_2P}\w,
\]
where $Q_0$ is a base point of the transform, $\w$ is the column formed by the base Prym differentials. By invariance of an integral with respect to a change of variables (in our case, given by $\tau_2$) $\int\limits_P^{\tau_2P}\w=\int\limits_{\tau_2P}^{P}\tau_2^*\w$. The relations $\tau_2^*\w=\w$ and $\int\limits_{\tau_2P}^{P}\w=-\int\limits_{P}^{\tau_2P}\w$ imply $\int\limits_{P}^{\tau_2P}\w=0$, hence $\A(\tau_2P)=\A(P)$.
\end{proof}
Let $\phi\in isoPrym_1(\Sigma)$ then $\A^{-1}(\phi)=P_1+\ldots+P_{2h_1}$. We will assume these points to be numbered in such way that $\tau_2(P_k)=P_{h_1+k}$, $k=1,\ldots,h_1$. Symmetric functions of  $P_1,\ldots,P_{h_1}$ are well-defined functions of $\phi$. It is our goal to find out a theta function formulae for a full independent set of such functions. We basically follow ideas by B.Dubrovin in \cite{Dubr_theta} (going back to Riemann), developed in relation to explicit reversion of the Abel map.

For any meromorphic $\tau_2$-invariant function $f$ on $\widehat\Sigma$ we consider $\s_f(\phi)=\sum_{P\in |D|} f(P)$ where $D=P_1+\ldots+P_{2h_1}$, $|D|=\rm{support}(D)$. Since $|D|=P_1+\ldots+P_{h_1}+\tau_2(P_1)+\ldots+\tau_2(P_{h_1})$, we have $\s_f(\phi)=2\sum_{i=1}^h f(P_i)$. Assuming $f$ to have no pole except at infinity we begin with the following relation close to the relation by Dubrovin (\cite[Eq. (11.23)]{Dubr_RegCh}, \cite[Eq. (2.4.29)]{Dubr_theta})  (it is the only difference with that relation that we take account of the above introduced multipliers $\epsilon_k$):
\begin{equation}\label{E:Dubr}
  \s_f(\phi)=c-\sum_{Q:f(Q)=\infty}\res_Q fd\ln F_\phi,
\end{equation}
where $c$ is constant in $\phi$, $F_\phi(P)=\theta(\A(P)-\epsilon^{-1}(\phi-\Delta))$, $(\epsilon^{-1}\psi)_j=\epsilon^{-1}_j\psi_j$, $j=1,\ldots,h_1$. For completeness, we reproduce the proof due to B.Dubrovin here, with minor changes due to the fact that we deal with Prymians and Prym theta functions.

It follows from the theorem on residues and \refL{inver} that
\[
   \s_f(\phi)=\frac{1}{2\pi i}\oint_{\partial\Sigma'}f(P)d\ln F_\phi(P) -\sum_{Q:f(Q)=\infty}\res_Q fd\ln F_\phi
\]
Observe that the first summand includes, in particular, the residues at the poles of $f(P)$, while the second summand compensates them, so that there is only the sum of the residues at zeroes of the function $F_\phi(P)=\theta(\A(P)-\epsilon^{-1}(\phi-\Delta))$ in remainder. Writing down the integral over the boundary of $\Sigma'$ as the sum of integrals along the cuts, similarly to the proof of Lemmas  \ref{L:num_zero} and \ref{L:inver}, we obtain
\begin{equation}\label{E:calc15}
 \frac{1}{2\pi i}\oint_{\partial\Sigma'}f(P)d\ln F_\phi(P) = \frac{1}{2\pi i}\sum_{j=1}^{\widehat{g}}\left(\oint_{a_j} + \oint_{b_j}         \right)f(P)(d\ln F_\phi(P)^+ - d\ln F_\phi(P)^-).
\end{equation}
On the right hand side, the $f(P)$ can be taken out of the brackets, for the reason it is a function on $\Sigma$, hence its evaluations on $a_j$ and on $a_j^{-1}$ (on $b_j$ and $b_j^{-1}$, resp.) coincide.

Then, making use of \refE{Rim_matr}, similarly to the proof of \refL{num_zero}, we conclude that the right hand side of the relation \refE{calc15} is equal to $\frac{1}{2\pi i}\sum_{j=1}^{\widehat{g}}\epsilon_j\oint_{a_j}f\w_j$ (where we set $\epsilon_{j+h_1}=\epsilon_j$ for $j=1,\ldots,g_1$), which does not depend on $\phi$. Thus we obtained \refE{Dubr}.

Below, we assume that $\Sigma$ is a branch covering of the Riemann sphere,  $\pi:\Sigma\to\mathbb{P}^1$ is the covering map, and $\pi$ is $\tau_2$-invariant. These assumptions are fulfilled in many cases, in particular for spectral curves of Hitchin systems. Let $x_i=\pi(P_i)$, $i=1,\ldots,h_1$. We take $f(P)=x^k$ where $x=\pi(P)$. Denote $\s_f$ by $\s_k$, then
\begin{equation}\label{E:sigm}
  \s_k(\phi)=x_1^k+\ldots +x_h^k,
\end{equation}
i.e. $\s_k(\phi)$ is the $k$th Newton polynomial in $x_1,\ldots ,x_{h_1}$. The relation \refE{Dubr} can be written down as
\begin{equation}\label{E:Dubr_k}
  \s_k(\phi)=c-\sum_{Q\in\pi^{-1}(\infty)}\res_Q x^kd\ln\theta(\mathcal{A}(P)-\epsilon^{-1}(\phi-\Delta)).
\end{equation}
Since $(d\mathcal{A})_i=\w_i$, we obtain that
\[
  d\ln\theta(\mathcal{A}(P)-\epsilon^{-1}(\phi-\Delta))=\sum_{i=1}^{h_1} (\partial_i\ln\theta(\mathcal{A}(P)-\epsilon^{-1}(\phi-\Delta))\w_i
\]
where $\w_i$ are given by the relations \refE{dprym}, $\partial_i$ stays for the derivative in the $i$th argument ($i=1,\ldots,6$). We choose an arbitrary point $Q_0\in\pi^{-1}(\infty)$ as a base point of the Abel--Prym transform. In a neighborhood of $Q_0$ we can consider $\mathcal{A}(P)$ as a small quantity, and expand  $(\ln\theta(\mathcal{A}(P)-\epsilon^{-1}(\phi-\Delta)))_i$ into a Tailor series. What we need to do after that, is to find out the sum of the terms of order $z^{2k-1}$ in the just obtained expansion, where $z$ is a local parameter in the neighborhood of the point $x=\infty$ ($x=z^{-2}$). Obviously, having been multiplied by $x^k=z^{-2k}$, this sum will give the required residue in \refE{Dubr_k}. As a result, we obtain the contribution of the point $Q_0$ into the expression \refE{Dubr_k} for $\s_k(\phi)$:
\begin{equation}\label{E:theta_sigma}
    \sum_{i=1}^{h_1}\sum_{1\le |j|\le 2k-1}
    {\varkappa_{ik}^j}D^j\partial_i\ln\theta(-\epsilon^{-1}(\phi-\Delta)),
\end{equation}
where $j=(j_1,\ldots,j_{h_1})$, $|j|=j_1+\ldots+j_{h_1}$,
\begin{equation}\label{E:theta_kappa}
  D^j=\frac{1}{j_1!\ldots j_h!}\frac{\partial^{|j|}}{\partial\phi_1^{j_1}\ldots\partial\phi_{h_1}^{j_{h_1}}},
  \quad {\varkappa_{ik}^j}=\sum_{l_i+\sum_{s=1}^{h_1}\sum_{p=1}^{j_s} l_{sp}=2k-1} \varphi_i^{(l_i)}\prod_{s=1}^{h_1}\prod_{p=1}^{j_s}\frac{\varphi_s^{(l_{sp}-1)}}{l_{sp}},
\end{equation}
$l_s$ and $\varphi_s^{(l_s)}$ are defined from the relation $\A_s(P)=\sum_{l_s\ge 1}\frac{\varphi_s^{(l_s)}}{l_s}z^{l_s}$ ($P=P(z)$).

Computation of the contribution of an arbitrary point $Q\in\pi^{-1}(\infty)$  differs in that we take a Tailor expansion in the small parameter $\A(P)-A(Q)$ which only results in addition of $\A(Q)$ to the argument of the theta function in \refE{theta_sigma}. Also the coefficients $\varphi_s^{(l_s)}$  and  $\varkappa_{ik}^j$ in \refE{theta_kappa} become dependent of $Q$  (and  acquire the index $Q$). Thus,
\begin{equation}\label{E:sigma_k_result}
    \s_k(\phi)=c-\sum_{Q\in\pi^{-1}(\infty)}\sum_{i=1}^{h_1}\sum_{1\le |j|\le 2k-1} {\varkappa_{Q,ik}^jD^j}\partial_i\ln\theta(\A(Q)-\epsilon^{-1}(\phi-\Delta)),
\end{equation}
The functions $\s_k(\phi)$, $k=1,\ldots,h_1$ give a full set of symmetric functions of $x$-coordinates of the points in $\A^{-1}(\phi)$. They determine $x_1,\ldots,x_{h_1}$ up to a permutation.
\section{Spectral curves of Hitchin systems}\label{S:Hitch}
For the classical definition of Hitchin systems we refer to \cite{Hitchin,BorSh,Sh_Bin}. Below, we use an alternative definition by means of giving the spectral curve and the Poisson bracket of a system in terms of separating variables \cite{Sh_FAN_2019,BorSh,Sh_Bin}. For Hitchin systems on hyperelliptic Riemann surfaces, with the structure groups $G=Sp(2n)$, $G=SO(2n)$ spectral curves are given by systems of equations in $\C^3$, of the form
\begin{equation}\label{E:spec_Hit}
  R(\l,x,y)=0,\quad y^2=P_{2g+1}(x),
\end{equation}
where
\begin{equation}\label{E:spec_Hit1}
   R(\l,x,y)=\l^{2n}+\sum_{j=1}^nr_j(x,y)\l^{2(n-j)}
\end{equation}
and
\begin{equation}\label{E:spec_Hit2}
   r_j(x,y)=\sum_{k=0}^{2j(g-1)} H^{(0)}_{jk} x^k +\sum_{s=0}^{(2j-1)(g-1)-2} H^{(1)}_{js}yx^s.
\end{equation}
In the case $G=Sp(2n)$ the spectral curve is generically nonsingular. In the case $G=SO(2n)$ the coefficient $r_n$ is a full square which means that, first, there are relations between $H^{(1)}_{ns}$, and, second, the curve has singularities in a generic position. In this case, here and below, we use  its normalization by default. We do not consider the systems with the structure group $SO(2n+1)$ here, for the reason they are locally isomorphic to the systems with the structure group $Sp(2n)$ \cite{Hitchin}. As for the systems with the structure group $GL(n)$, we refer to \cite{Sh_FAN_2019,BorSh,Sh_Bin}. Briefly speaking, their spectral curves can be obtained by plugging $n$ instead $2n$, and  $j$ instead $2j$ in \refE{spec_Hit}--\refE{spec_Hit2}, where $n$ is arbitrary, $j=1,\ldots,n$.

A curve of the form \refE{spec_Hit}--\refE{spec_Hit2} is invariant with respect to the holomorphic involution $\tau_1:\l\to -\l$. In the case $G=Sp(2n)$, the base of  its Prym differentials is given by the differentials $\frac{\partial R/\partial H_{jk}^{(0)}}{\partial R/\partial\l}\frac{dx}{y}$, $\frac{\partial R/\partial H_{js}^{(1)}}{\partial R/\partial\l}\frac{dx}{y}$ where $j,k,s$ vary within the same limits as in \refE{spec_Hit1}, \refE{spec_Hit2}. In the case $G=SO(2n)$, accordingly $r_n(x,y)=q(x,y)^2$, the list of base differentials is the same for $j<n$, while for $j=n$ it is formed by the differentials $\frac{x^kqdx}{R_\l'\, y } \ (0\le k\le n(g-1))$ and $\frac{x^sqdx}{R_\l'}\ (0\le s\le {(n-1)(g-1)-2})$.

Let $h_1=\dim Prym_1$ be dimension of the Prymian of the curve with respect to the involution $\tau_1$. Then $h_1=(\dim G)(g-1)$ \cite{Hitchin}. In our case, this relation can be easily obtained by means of straightforward counting the base Prym differentials. Separating variables of the system are given by sets of triples of complex numbers $(x_i,y_i,\l_i)$ ($i=1,\ldots,h_1$), each one satisfying the relation $y_i^2=P_{2g+1}(x_i)$  \cite{Sh_FAN_2019,BorSh,Sh_Bin}. The Poisson bracket in the separating variables is given by the relations $\{\l_i,x_j\}=y_i\d_{ij}$.

Observe that the spectral curves of Hitchin systems are coverings of $\mathbb{P}^1$, in a natural way ($\pi : (x,y,\l)\to x$).
If $H^{(1)}_{js}=0$ for all $j,s$ then the curve \refE{spec_Hit}--\refE{spec_Hit2} possesses also the involution $y\to -y$. In this case we set $\tau_2:\l\to -\l,y\to -y$. For only two classical groups, on only genus 2 Riemann surfaces the corresponding Hitchin systems possess the property that $H^{(1)}_{js}$ vanish for all  $j,s$. These groups are $SL(2)$ and $SO(4)$. For all other systems that condition is fulfilled only for some degenerated curves. We shall consider spectral curves of the form \refE{spec_Hit}--\refE{spec_Hit2} fulfilling the property, i.e.  invariant with respect to the involutions $\tau_1$, $\tau_2$, and, moreover, such that this pair of involutions is of the first type. In this case, all results of the previous sections are applicable. In particular, relations \refE{sigma_k_result} enable us to find out the coordinates $x_1,\ldots,x_{h_1}$ of the preimage of a point $\phi\in isoPrym_1$ in terms of Prym theta functions, up to a permutation. Then we can find out the corresponding $y_i$, $\l_i$ from the system of equations \refE{spec_Hit}--\refE{spec_Hit2}, but not uniquely. However, if we fix a certain branch of the covering $\pi : \Sigma\to\mathbb{P}^1$ (for every $i=1,\ldots,h_1$) then by taking the solution for $y_i$, $\l_i$ on that branch, the reversion procedure can be made unambiguous (outside the branch points). These considerations are quite enough to locally construct a trajectory of a Hitchin system with a given Hamiltonian, and given initial condition. Indeed, giving the initial condition, i.e. a point
$\ga_0=\{(x_1^0,y_1^0,\l_1^0),\ldots,(x_{h_1}^0,y_{h_1}^0,\l_{h_1}^0) \}$ outside the branch points, determines a set of branches of the spectral curve. Let $\phi_0=\A(\ga_0)$. On $isoPrym_1$, the trajectory has the form $\phi=It+\phi_0$ where the (vector valued) coefficient $I$ depends on the Hamiltonian only. By pluggin it into the just constructed (in terms of the Prym theta function) map $\ga=\A^{-1}(\phi)$, we resolve the problem in a neighborhood of the point~$\ga_0$.
\section{Examples of curves with a pair of involutions of the first type
}\label{S:Examples}
In this section, we will show with examples that all cases of \refL{k1t} indeed occur. All our examples are related to spectral curves of Hitchin systems with structure groups $SL(2)$, $SO(4)$, $Sp(4)$, except for one example, namely for the spectral curve of the Kovalewski system.

\begin{example}\label{E:SL2gen2}
The Hitchin system with the structure group $SL(2)$ on a genus 2 curve. The spectral curve $\Sigma$ is given by the system of equations
\begin{equation}\label{E:SL2_1}
  R(\l,x)=\l^2+r(x)=0,\quad y^2=P_5(x)
\end{equation}
where $r(x)=H_0+H_1x+H_2x^2$. As above, let $\widehat{g}$ stay for the genus of the curve $\Sigma$. For the system in question  $\widehat{g}=5$, $\dim Prym_1=3$ \cite{BorSh}. A full set of Prym differentials (from now on we assume them to be holomorphic, by default) is given by the list $\frac{dx}{\l y}$, $\frac{xdx}{\l y}$, $\frac{x^2dx}{\l y}$. They are obviously invariant with respect to $\tau_2$. Vice versa, the $\tau_1$-invariant differentials given by $\frac{dx}{y}$ and $\frac{xdx}{y}$ are antiinvariant with respect to $\tau_2$. This implies $h_1=3$, $h_2=2$. Since $h_1>h_2$, the case $2^\circ$ of \refL{k1t} takes place. According to \refT{birat} $Prim_1\simeq Jac(\Sigma_2)/\Z_2$, and according to \refL{k1t} $g_2=3$. The above obtained results give solutions of the system in terms of Prym theta functions in the dimension 3. The system has been a subject of the works \cite{Gaw,Previato,Geemen_Jong} which finally resulted in its solution in theta functions. The proposed here derivation of solutions from the general results of sections \ref{S:ogr}, \ref{S:theta-form} is all-sufficient, much shorter, and we believe, more direct and transparent. However, related algebraic-geometric and field theoretic results of the works \cite{Gaw,Previato,Geemen_Jong} remain outside the scope of our approach. Observe also that it is claimed in \cite{Gaw} that invariant tori of the system in question are Jacobians of genus 3 hyperelliptic curves.
\end{example}
\begin{example}\label{Ex:SO4_2}
The Hitchin system with the structure group $SO(4)$ on a genus 2 curve. According to \cite{BorSh} the spectral curve of the system is given by equations
\begin{equation}\label{E:SO4_2}
  \l^4+p(x)\l^2+q^2(x)=0,\quad y^2=P_5(x)
\end{equation}
where $p$ and $q$ are quadratic polynomials. A base of holomorphic Prym differentials on (the normalization of) $\Sigma$ is given by the following list \cite{BorSh}:
\begin{equation}\label{E:dprym}
\begin{aligned}
       &\omega^{(0)}_i = \frac{x^{i-1}q(x)dx}{y\l(4\l^2 + 2p(x))},~ i = 1, 2, 3,\\
       &\omega^{(1)}_i = \frac{\l^2x^{i-4}dx}{y\l(4\l^2 + 2p(x))},~ i = 4, 5, 6.
\end{aligned}
\end{equation}
It is known \cite{Hitchin} that the spectral curve of a Hitchin system with the structure group $SO(2n)$ is unramified over its quotient by the involution $\l\to -\l$, i.e. over $\Sigma_1$ in our case. For the systems with the group $SO(4)$ on hyperelliptic curves $y^2=P_5(x)$ (of genus 2) a genus of the spectral curve is equal to 13 (i.e. odd), while the involutions $\tau_1:\l\to -\l$, $\tau_2 :\l\to -\l, y\to -y$ form the pair of the first type \cite{Sh_SO4}, as it easily follows from the form of the Prym differentials \refE{dprym}. Hence either the case $3^\circ$, or the case $4^\circ$ of \refL{k1t} takes place.

In both cases by $n_1=0$ we have $\widehat{g}=2g_1-1$, i.e. $g_1=7$, $h_1=\widehat{g}-g_1=6$.

In the case $3^\circ$ it would be $h_2=g_1=7$. But in fact we have 6 holomorphic differentials on $\Sigma_1$, which become $\tau_1$-symmetric and $\tau_2$-antisymmetric after pull back to $\Sigma$, and one more differential symmetric with respect to both involutions. Due to this behavior with respect to the involutions, all 13 differentials are linear independent. Hence $h_2=6$, and we have the case $4^\circ$ of the Lemma.

An explicit form of the above mentioned basis holomorphic differentials on~$\Sigma_1$ is as follows.  The curve $\Sigma_1$ is given by the equations $R_1=0, y^2=P_5$ where $R_1=\mu^2+\mu p+q$, $\mu=\l^2$. If $p=H_0+H_1x+H_2x^2$, $q=H_3+H_4x+H_5x^2$ then we have 6 holomorphic differentials of the form $\frac{\partial R_\s/\partial H_j}{\partial R_\s/\partial \mu}\frac{dx}{y}$, $j=0,1,\ldots,5$:
\begin{equation}\label{E:dsym}
\begin{aligned}
       &\omega^{(0)}_i = \frac{x^{i-1}q(x)dx}{y(2\mu + p(x))},~ i = 1, 2, 3,\\
       &\omega^{(1)}_i = \frac{\mu x^{i-4}dx}{y(2\mu + p(x))},~ i = 4, 5, 6.
\end{aligned}
\end{equation}
They are obviously symmetric with respect to $\tau_1$, and antisymmetric with respect to~$\tau_2$. Besides, there is one more differential $dx/\mu$ symmetric with respect to both involutions. At infinity $\l\sim z^{-2}$, hence $\mu\sim z^{-4}$. Further on, $x\sim z^{-2}$, $dx\sim z^{-3}dz$, and $dx/\mu\sim zdz$, i.e. it is holomorphic. Observe that $\mu=0$ is a smooth point of the curve $\Sigma_1$ unless $p$ and $q$ have common zeroes, and it is not a branch point. Hence the differential $dx/\mu$ is holomorphic at $\mu=0$ due to Proposition 3.1 \cite{BorSh}.

Results of sections \ref{S:ogr}, \ref{S:theta-form} give solution of the system in Prym theta functions in the dimension 6. For a detailed presentation of resolving the system we refer to \cite{Sh_SO4}, where it was solved for the first time.
\end{example}
\begin{example}\label{Ex:Kowalevski}
Kovalewski system. According to \cite[\S 5.13]{BBT}, the spectral curve of the system is a normalization of a flat algebraic curve  of the form $f(\mu^2,\l^2)=0$, its genus is equal to 5. The involutions are as follows: $\tau_1:\l\to -\l$, $\tau_2: \mu\to -\mu$.
The curve $\Sigma_1$ (often denoted by $C$ for the Kovalewski problem) is of genus 3, $\Sigma_2$ is of genus 2, i.e. $h_1=g_2=2$ and $h_2=g_1=3$. The Poisson bracket has the form $\{\l_j,\mu_k\}=-i\mu_k\d_{jk}$ in the separation variables \cite[Ch. 5, \S 4]{Ts}. Hence, we have the case $3^\circ$ of the Lemma ($h_1<h_2$). In particular, $\Sigma_2$ is a hyperelliptic genus 2 curve, and by \refT{birat} $Prym_1\cong Jac_2$. The results of \cite[Sect. 4.3]{Sh_Bin} (similar to the results of \refS{theta-form}, but for Jacobians) provide solutions in theta functions of genus 2. It completely meets the original Kovalewski solution given in the same terms.
\end{example}
\begin{example}\label{Ex:SL2_gen3}
Spectral curve of an $SL(2)$ Hitchin system on a genus~3 Riemann surface:
\begin{equation}\label{E:SL2_1}
  R(\l,x)=\l^2+r(x)=0,\quad y^2=P_7(x).
\end{equation}
In general $r(x)=\sum_{i=0}^{4}H_ix^i+H_5y$ \cite{Sh_FAN_2019,BorSh}, but we assume that
\begin{equation}\label{E:SL2_2}
 r(x)= H_0+H_1x+H_2x^2+H_3x^3.
\end{equation}
In a generic position, the curve \refE{SL2_1}, \refE{SL2_2} is non-singular, since the equations for singular points
\[
  R'_\l=2\l=0,\quad R'_x=H_1+2H_2x+3H_3x^2=0.
\]
descend to the equations $r(x)=0$ and $r'(x)=0$, which are incompatible.

The equations for branch points are as follows: $\l=0$, which implies $r(x)=0$. The last equation has 3 roots in general. Due to the symmetry $y\to -y$ we obtain 6 branch points. By the Riemann--Hurwitz formula for a 2-fold covering, and $g=3$ we have
\[
    2\widehat{g}-2= 2(2\cdot 3-2)+ 6
\]
which implies $\widehat{g}= 8$. The total number of linear independent differentials is equal to 8:
\begin{itemize}
\item
4 differentials of the form $\frac{\partial R/\partial H_j}{ R'_\l} \frac{dx}{y}$, $j=0,1,2,3$: $\frac{1}{\l}\frac{dx}{y}$, $\frac{x}{\l}\frac{dx}{y}$, $\frac{x^2}{\l}\frac{dx}{y}$, $\frac{x^3}{\l}\frac{dx}{y}$

\item
3 differentials pulled back from the base curve: $\frac{dx}{y}$, $\frac{xdx}{y}$, $\frac{x^2dx}{y}$,

\item
and the differential $\frac{dx}{\l}$.
\end{itemize}
\noindent
The orders of the differentials in the first quadruple, in a local coordinate at infinity, are equal to $z^7dz$, $z^5dz$, $z^3dz$, $zdz$, respectively. For the second triple they are as follows: $z^4dz$, $z^2dz$, $dz$, and for the last one again $dz$ (we use here that $x\sim z^{-2}$, $\l\sim z^{-3}$). Linear independence of the first seven differentials follows by that their orders at infinity are different, while linear independence between them, and the last differential $\frac{dx}{\l}$ follows from their different behaviour with respect to the symmetries $\tau_1$ and $\tau_2$: the first group is skew-symmetric with respect to $\tau_1$, and symmetric with respect to $\tau_2$, while the second group vice versa. As for the last differential, it is skew-symmetric with respect to both involutions. For the holomorphy we refer to  \cite{Sh_FAN_2019,BorSh}.

Presence of a basis differential skew-symmetric with respect to both involutions means that $\tau_1$, $\tau_2$ is not a pair of the first type in general. Assume that the polynomial $r(x)$ has one double zero, and one simple zero:
\[
 r(x)=a(x-b)^2(x-c),\quad b\ne c.
\]
Generically such curve has two singular points, and two branch points. Indeed, the singular points satisfy to the system of equations $\l=0$, $r'(x)=0$. The last can be written down as  $a(x-b)(3x-b-2c)=0$. In general, from the two its solutions only one, namely $x=b$, satisfies to the equation of the curve $\Sigma$. Taking account of the symmetry $y\to -y$ we obtain two singular points. The branch points can be found out from the equations $\l=0$, $r(x)=0$. There appears only one solution different from singular points, namely $x=c$. By the symmetry in $y$, we obtain two branch points: $x=c,y=\pm \sqrt{P_7(c)}$.

By the Riemann--Hurwitz formula we count a genus of normalization of the covering curve:
\[
    2\widehat{g}-2= 2(2\cdot 3-2)+ 2,
\]
which gives $\widehat{g}=6$. Hence, there are six basis holomorphic differentials. For sure, the differentials with the asimptotics $z^5dz$, $z^3dz$, $zdz$ from the first group are among them. Throwing away of any of them would result in an incomplete system in general. It is sufficient to add the differentials pulled back from the base in order to obtain a full system of holomorphic differentials:
\[
  \frac{xdx}{\l y},\quad\frac{x^2dx}{\l y},\quad\frac{x^3dx}{\l y}, \quad \frac{dx}{y},\quad\frac{xdx}{y},\quad\frac{x^2dx}{y}.
\]
The first three of them are $\tau_1$-antiinvariant, and invariant with respect to $\tau_2$, while the second triple vice versa. Hence these involutions form a pair of the first type.

This example corresponds to the case $1^\circ$ of \refL{k1t}: an even genus of the covering, and two branch points. The quotient by $\tau_2$ coincides with the base curve itself.
\end{example}
\begin{example}\label{Ex:arb_gen}
We generalize \refEx{SL2_gen3} onto the case of an arbitrary genus base curve (the structure group still is $SL(2)$). The main goal of the example is to show how one can obtain particular solutions of Hitchin systems on an arbitrary genus base curve by means of the results of \refS{theta-form}.

Let $\Sigma$ be the spectral curve, $\widehat{g}=genus(\Sigma)$, $\Sigma_0: y^2=P_{2g+1}(x)$ be a base curve, $\Sigma_{1,2}$, $g_{1,2}$ are as above. In order the covering $\Sigma\to\Sigma_0$ had two branch points, it is necessary, by the Riemann--Hurwitz formula, that $2\widehat{g}-2= 2(2g-2)+ 2$, which implies $\widehat{g}=2g$. In general, for $SL(2)$
\begin{equation}\label{E:r2_sl2}
   r(x,y)=\sum_{k=0}^{2(g-1)} H^{(0)}_{k} x^k +\sum_{s=0}^{g-3} H^{(1)}_{s}yx^s.
\end{equation}
We consider the case $H^{(1)}_{s}=0$, $s=0,\ldots,g-3$, and denote the coefficients $H^{(0)}_{k}$, $k=0,\ldots,2(g-1)$ by $H_k$. Then $r(x)=\sum_{k=0}^{2(g-1)} H_k x^k $. This polynomial has $2(g-1)$ roots. To have two branch points, it is necessary that all roots are double zeroes (they will give simple singular points), except for one. Then the degree of the polynomial must be odd, and we set $H_{2(g-1)}=0$. Thus the curve in question has equation \refE{SL2_1} where
\begin{equation}\label{E:r2_sl2_2}
   r(x)=\sum_{k=0}^{2g-3} H_k x^k .
\end{equation}
The differentials
\[
   \frac{x^sdx}{\l y}\ (s=0,\ldots,2g-3),\quad \frac{x^qdx}{y}\ (q=0,\ldots,g-1)
\]
are holomorphic. From the first group, $g$ differentials with minimal orders at infinity must be left. All differentials of the first group are antisymmetric with respect to $\tau_1$ and symmetric with respect to  $\tau_2$, and in the second group vice versa, i.e. we have obtained the pair of involutions of the first type. This series of examples gives solutions depending on $g$ integrals from $3g-3$ possible. We stress that the relations between the coefficients $H_k$ in the equation of the spectral curve must ensure the multiplicity two for all roots of the polynomial $r(x)$, except for one which must be simple.
\end{example}

We conclude with obtaining some particular solutions of Hitchin systems with the structure group $Sp(4)$ on genus two curves. The corresponding spectral curves give us examples of the cases $1^\circ$ and $2^\circ$  of \refL{k1t}.

Generically, the spectral curve of a Hitchin system with the structure group $Sp(4)$ on a genus two curve is given by the pair of equations \cite{BorSh}
\begin{equation}\label{E:ex_Sp4_1}
  R(\l,x)=0,\quad y^2=P_5(x),
\end{equation}
where
\begin{equation}\label{E:ex_Sp4_2}
\begin{aligned}
    R(\l,x)&=\l^4+\l^2p(x)+q(x), \\
    p(x)&=H_0+H_1x+H_2x^2,\\
    q(x)&=H_3+H_4x+H_5x^2+H_6x^3+H_7x^4+yH_8+xyH_9.
\end{aligned}
\end{equation}
We set $H_8=H_9=0$. We  seek for the cases when the pair of involutions $\tau_1,\tau_2$ belongs to the first type.
\begin{example}\label{Ex:Sp4}
Assume that $\Sigma$ is non-singular, and ramified over $\Sigma_1$ with order (degree of the branch divisor) equal to~4, i.e. the case $2^\circ$ of \refL{k1t} takes place. The system of equations for the branch points is as follows:
\begin{equation}\label{E:branch}
\left \{
\begin{array}{c}
     R'_{\l}(\l, x) = \l(4\l^2 + 2p)=0,\\
     R(\l, x) = 0,\quad y^2=P_5(x).
     \end{array}
\right.
\end{equation}
The solutions $\l=0$, $q(x)=0$ correspond to branch points $\Sigma$ over $\Sigma_1$ while the others correspond to the branch points of  $\Sigma_1$ over the base curve. Taking account of the fact that for every solution of \refE{branch} there is a symmetric one with $y$ of an opposite sign, there can be four branch points over $\Sigma_1$ if $\deg q=2$ only, i.e. $H_6=H_7=0$. The branch points with $\l\ne 0$ satisfy the equation $\l^2=-p/2$, which having been plugged into the equation of the curve gives a degree four equation $p(x)^2=4q(x)$. Since every $x$ is assigned with two values of $y$, and  four values of $\l$, it gives 16 branch points.  The total number of the branch points of the covering $\Sigma\to\Sigma_0$ is equal to 20, which implies $\widehat{g}=15$ by the Riemann--Hurwitz formula. We can take the following differentials as basis Prym differentials with respect to $\tau_1$:
\begin{equation}\label{E:BPrym_Sp4}
  \frac{\l^2x^idx}{\l y(4\l^2+2p)}\ (i=0,1,2);\quad \frac{x^idx}{\l y(4\l^2+2p)}\ (i=0,1,2,3,4)
\end{equation}
( all of them are of the form $\frac{\partial R/\partial H_j}{R'_\l}\frac{dx}{y}$, $j=0,1,\ldots,7$). It is easy to check that they are holomorphic at infinity, it follows from $\l\sim z^{-2}$, $x\sim z^{-2}$, $y\sim z^{-5}$ where $z$ is a local parameter at infinity \cite{Sh_FAN_2019}.

The genus of $\Sigma_1$ can be found out from the relation $2\widehat{g}-2=2(2\cdot g_1-2)+4$, as well as from the relation $2g_1-2=2(2\cdot 2-2)+8$, the last for the genus of the 2-fold covering ${\Sigma}_1\to{\Sigma}_0$ with eight branch points (coming from the 16 branch points with $\l\ne 0$ glued in pairs under projection $\Sigma\to\Sigma_1$). We obtain $g_1=7$ as a result.

Thus, there are seven independent holomorphic $\tau_1$-symmetric differentials on $\Sigma$, hence the differentials \refE{BPrym_Sp4} form a base of holomorphic Prym differentials on that curve. These differentials are invariant with respect to the involution $\tau_2$. For this reason, $\tau_1,\tau_2$ is a pair of involutions of the first type. We have proven that the example in question corresponds to the case $2^\circ$ of \refL{k1t}. However, in this example inversion of the Abel--Prym map can not be used for solution of the integrable system, because the dimension of the isoPrymian is equal to 8 while the number of degrees of freedom of the system is six.
\end{example}
The following example does not have this drawback. The curve $\Sigma$ has two singular points (before normalization), and two branch points over $\Sigma_1$ in this example.
\begin{example}\label{Ex:Sp4_2}
Let $q(x)=H_3+H_4x+H_5x^2+H_6x^3+H_7x^4$ to have one double, and two simple roots (that is one singular point on $\Sigma$, subject to resolution). Thus, there are two branch points with $\l=0$ (i.e. branch points of the covering $\Sigma\to\Sigma_1$). Then, by the Riemann--Hurwitz formula, $2\widehat{g}=2g_1$. The total number of branch points of the covering $\Sigma\to\Sigma_0$ is equal to 18, and $\widehat{g}=14$, which implies $g_1=7$. Thus a base of holomorphic differentials on  $\Sigma$ consists of the seven Prym differentials (which are selected from the eight differentials $\frac{\partial R/\partial H_j}{R'_\l}\frac{dx}{y}$), and of the seven $\tau_1$-symmetric differentials. This corresponds to the case $1^\circ$ of \refL{k1t}, and the dimension of the isoPrymian is equal to the number of degrees of freedom of the system. Indeed, the integrals of the system are as follows: the three roots and the coefficient at the highest degree of the polynomial $q(x)$, and three coefficiients of the polynomial $p$.
\end{example}
If to consider the curve given by equations \refE{ex_Sp4_1}, \refE{ex_Sp4_2} with
\[
   q(x)=H_3+H_4x+H_5x^2+H_6x^3+H_7x^4,
\]
under assumption that all solutions to the system \refE{branch} with $\l=0$ correspond to singular points (i.e. satisfy the equation $R'_x=0$, in addition), then it is easy to see that this amounts in the system of equations $q(x)=0, q'(x)=0$. Hence the singular points are multiple roots of the polynomial $q(x)$. Generically, their multiplicities are equal to 2, hence the polynomial $q$ is a full square of a quadratic polynomial. Thus we returned to the Hitchin system with the structure group  $SO(4)$ on a genus 2 curve (\refEx{SO4_2} above).


\bibliographystyle{amsalpha}

\end{document}